\documentclass[12pt]{amsart}

\usepackage{amsfonts,color}
\usepackage{latexsym,amssymb}
\usepackage{amsmath}
\usepackage[utf8]{inputenc}
\usepackage{dsfont,comment,cases}
\usepackage{empheq,enumitem}     

\newtheorem{theorem}{Theorem}[section]

\newtheorem{corollary}[theorem]{Corollary}

\newtheorem{remark}[theorem]{Remark}

\theoremstyle{definition}
\newtheorem{example}[theorem]{Example}
\theoremstyle{plain} 

\usepackage{soul,xcolor} 

\usepackage[bookmarksnumbered,colorlinks, linkcolor=blue, citecolor=red, pagebackref, bookmarks, breaklinks]{hyperref}

\usepackage[hyperpageref]{backref}

\newcommand{\dint}{\int_{\Omega}}

\catcode`@=11 \@addtoreset{equation}{section} \catcode`@=12

\usepackage[left=2.2cm,right=2.2cm,top=2.6cm,bottom=2.4cm]{geometry}

\begin{document}	
\title[Weighted Hardy-Sobolev type inequalities]{Weighted Hardy-Sobolev type 
inequalities with boundary terms
}
	
	\author[Do o]{J. M. do Ó}
	\address{Universidade Federal da Para\'iba, Departamento de Matem\'atica
 \newline\indent
 58051-900 Jo\~{a}o Pessoa-PB, Brazil}
	\email{\href{mailto:jmbo@mat.ufpb.br}{jmbo@mat.ufpb.br}}
	
	\author[Furtado]{M.F. Furtado}
	\address{Universidade de Bras\'ilia, Departamento de Matem\'atica
 \newline\indent
 70910-900 Bras\'ilia-DF, Brazil}
	\email{\href{mailto:mfurtado@unb.br}{mfurtado@unb.br}}

\author[Medeiros]{E.S. Medeiros}
\address{Universidade Federal da Para\'iba, Departamento de Matem\'atica
\newline\indent
58051-900 Jo\~{a}o Pessoa-PB, Brazil}
\email{\href{mailto: everaldo@mat.ufpb.br}{everaldo@mat.ufpb.br}}

\author[J. Ratzkin]{J. Ratzkin}
\address[J. Ratzkin]{Department of Mathematics,
	Universit\"{a}t W\"{u}rzburg
	\newline\indent
	97070, W\"{u}rzburg-BA, Germany}
\email{\href{mailto:jesse.ratzkin@uni-wuerzburg.de}{jesse.ratzkin@uni-wuerzburg.de}}

	\subjclass{35J55, 35J50, 37K05, 35J62}
	\keywords{Weighted Sobolev embedding, Hardy-type inequality, Trudinger-Moser inequality.  }
		
	\begin{abstract}
		In this paper, we establish a new class of weighted Hardy–Sobolev type inequalities 
        under mild monotonicity assumptions on the weight function. As a consequence, we 
        derive the corresponding weighted Sobolev and trace-type inequalities. These 
        results play an important role in the analysis of elliptic problems with Neumann 
        or Robin boundary conditions in unbounded domains.
    \end{abstract}

	\maketitle
	
	\begin{center}
		\begin{minipage}{8cm}
		\footnotesize
		\tableofcontents
		\end{minipage}
	\end{center}

\section{Introduction and Statement of Results} 

Functional inequalities play a critical role in modern analysis and 
its applications, with the famous inequalities of Hardy and of Sobolev serving 
as models. The classical Sobolev inequality in 
$\mathbb{R}^N$ takes the form 
\begin{equation}\label{Sobolev-Classical} 
\mathcal{S}_{N,p} 
\left (\int_{\mathbb{R}^N} |u|^{\frac{Np}{N-p}} dx \right )^{\frac{N-p}{N}} \leq \int_{\mathbb{R}^N} |\nabla u|^p dx, 
\quad \forall \, u \in C_0^\infty (\mathbb{R}^N),
\end{equation} 
where $\mathcal{S}_{N,p}>0$ is a positive constant and $1\leq p <N$. In \cite{Hardy1}, G. H. Hardy 
proved that  
$$
\left(\frac{p-1}{p}\right)^p \int_0^\infty \frac{|u|^p}{|t|^p} dt\leq \int_0^{\infty} |u'|^p
dt, \quad \forall \,u\in C_0^\infty (0,+\infty),
$$
which one can then extend to 
\begin{equation}\label{Hardy-Classical}
		\left(\frac{p-1}{p}\right)^p\int_{\mathbb{R}^N_+} \frac{|u|^p}{x_N^{p}}dx
        \leq\int_{\mathbb{R}^N_+} |\nabla u|^pdx,\quad \forall \,u\in C_0^\infty (\mathbb{R}^N_+), \, p>1,
	\end{equation}
in the half-space $\mathbb{R}^N_+ = \{ x= (x', x_N) \in \mathbb{R}^N: x_N 
\geq 0\}$. In their classical form, Hardy inequalities express the fact that 
certain weights can be controlled by the gradient term, revealing a delicate balance between 
integrability and singularity. It is worth remarking here that \eqref{Hardy-Classical}
holds only when the function $u$ vanishes along the hyperplane $\{ x_N = 0\}$; 
see Remark~\ref{Remak-Singular} for further commentary. 

In the book \cite{Mazya}, the author combines the Sobolev embedding with Hardy inequality to 
obtain a single statement. Precisely, it was shown that, for those smooth functions 
on the upper half-space
$\mathbb{R}^N_+$ that vanish on the boundary $x_N =0$,  there holds the sharp Hardy-Sobolev-Maz’ya inequality:
\begin{equation}\label{Mazya}
\left(
\int_{\mathbb{R}^N_+} 
|u|^{\frac{2N}{N-2}} \, dx
\right)^{\!\frac{N-2}{N}}
\leq 
C \int_{\mathbb{R}^N_+} 
\left( 
|\nabla u|^2 - \frac{|u|^2}{4x_N^2}
\right) \, dx,
\qquad \forall\,
u \in C_0^\infty(\mathbb{R}^N_+).
\end{equation}
This inequality, its generalizations to different powers of the gradient, and the 
study of optimal 
constants have attracted considerable attention in recent years.  
Further extensions of the Hardy–Sobolev–Maz’ya inequality~\eqref{Mazya} to convex domains have 
also been investigated; see, for instance,~\cite{Filippas-Terticlas1,  Frank-Loss} 
and the references therein. 

These inequalities have been widely used to study existence and Liouville-type results. For 
instance, in~\cite{zbMATH05077864,Athanasios}, the authors, motivated by the 
study of Liouville-type results, 
considered the problem of determining the optimal constant in this inequality.

In recent years, several extensions and refinements of Hardy’s inequality have 
been established, 
including weighted and anisotropic versions, inequalities with remainder terms, and 
Hardy–Sobolev-type inequalities involving additional effects of integrability. Such weighted 
inequalities have found important applications in the study of elliptic and parabolic equations 
with singular coefficients, in potential theory, and in spectral problems associated with 
degenerate operators. For \( u \in C_0^\infty(\Omega) \) and \( 1 < p < N \), a weighted 
Hardy-Sobolev inequality states that
\begin{equation}\label{Geral}
\int_{\Omega} W_1 |u|^p \, dx \leq C \int_{\Omega} W_2 |\nabla u|^p \, dx,
\end{equation}
where \( \Omega \subset \mathbb{R}^N \) (\( N \ge 3 \)) and the weights $W_1$ and 
$W_2$ satisfy appropriate conditions; see, for example,~\cite{Brezis1, Opic} and 
references therein. To give one further example, we cite 
\cite[Theorem 6.9]{Ambrosio-Dipierro}, where the authors proved that 
\begin{equation}\label{Serinha}
		\left|\frac{\gamma-p+1}{p}\right|^p\int_{\mathbb{R}^N_+} \frac{|u|^p}{x_N^{p-\gamma}}dx
        \leq\int_{\mathbb{R}^N_+} x_N^\gamma|\nabla u|^pdx,\quad \forall \,u\in C_0^\infty (\mathbb{R}^N_+),
	\end{equation}
holds for any $p>1$ and $\gamma\in \mathbb{R}$. Setting $\gamma=0$ we 
recover \eqref{Hardy-Classical}.

Our goal in the present manuscript is to prove weighted Hardy-Sobolev 
inequalities in the domain above a graph in $\mathbb{R}^N$, both for 
increasing and for decreasing weights. We also explore some, but not all, 
of the consequences of these weighted inequalities. Our results here include 
extending \eqref{Serinha} to functions that do not vanish 
along $\{ x_N = 0\}$ provided $\gamma>p-1$ (see Corollary \ref{Singular-1}); 
additionally we provide a counter-example in Remark~\ref{Remak-Singular}
demonstrating that the 
condition $\gamma > p-1$ is sharp. 
The present work includes that 
of \cite{Abreu-Felix-Medeiros,Abreu-Furtado-Medeiros}, as well as other works, 
as special cases. Furthermore, we provide examples showing our conditions on 
the weight functions cannot be weakened, and so our results are sharp. 
We also weaken the 
regularity hypotheses of, for instance, \cite{Lewis}. It worth mentioning that, 
while the weight 
should vanish on the boundary to mimic the classical Hardy inequality, our inequalities 
apply to functions that do not themselves vanish on the boundary, in contrast 
to, for instance, \cite{Cossetti-DArca}.

Let $N \geq 2$ and $p > 1$. As 
in  \cite{Berestycki-Caffarelli-Nirenberg1, Berestycki-Caffarelli-Nirenberg2}, we consider a 
continuous function $\psi:\mathbb{R}^{N-1} \to \mathbb{R}$ and set
 	$$
	\Omega:=\left\{x=(x', x_N) \in \mathbb{R}^{N} : 
    x'\in \mathbb{R}^{N-1}, \;  x_N>\psi(x')\right\}.
	$$
We recover the half-space $\mathbb{R}^N_+$ by setting $\psi\equiv 0$. 

Throughout the paper, we suppose that $W$ and its weak derivative with respect to $x_N$ 
satisfy the basic assumption 
\begin{equation} \label{W_0} \tag{$W_0$} W:\overline{\Omega} \rightarrow [0,\infty), 
\quad W \in L^1_{\mathrm{loc}}(\Omega),\quad W_{x_N} := \frac{\partial W}{\partial x_N}  
\in L^1_{\mathrm{loc}}(\Omega). \end{equation}

We present our results in several cases, according to the monotonicity conditions on 
the weight function $W$.

\subsection{Weighted Hardy inequality and consequences: increasing weight case}

In our initial results, we focus on weights that are increasing in the $x_N$-direction. 
We impose the following assumption:
\begin{equation}\label{W_1^+}\tag{$W_1^+$}
\frac{W^p}{[W_{x_N}]^{p-1}} \in L^1_{\mathrm{loc}}(\Omega) ,\quad W_{x_N} > 0 
\text{ a.e. in } \Omega.\end{equation} 
Under these conditions, we establish the following result.

\begin{theorem}[Hardy inequality of type I] 
\label{Hardy-Grafico}
Suppose \eqref{W_0} and \eqref{W_1^+}. If $p>1$ then,  for any $u \in C_0^\infty(\mathbb{R}^N)$, the following inequality is satisfied:
\begin{equation}\label{Hardy-Geral-Grafico}
\int_{\Omega} W_{x_N} |u|^p \, dx 
+ p \int_{\mathbb{R}^{N-1}} W(x',\psi(x')) |u(x',\psi(x'))|^p \, dx' 
\leq p^p \int_{\Omega} \frac{W^p}{[W_{x_N}]^{p-1}} |\nabla u|^p \, dx.
\end{equation}
\end{theorem}

The key distinction of the domains in Theorem~\ref{Hardy-Grafico} is that their 
boundaries are merely continuous, unlike the sufficiently smooth boundaries required in 
earlier work (e.g., \cite[Lemma~2]{Lewis}). This lack of smoothness means the outward 
normal vector is often undefined (e.g., for Weierstrass-type function graphs), 
rendering Green's First Identity and the Divergence Theorem inapplicable. We can also 
prove Theorem \ref{Hardy-Grafico} for more general domains, including non-convex sets 
(see Corollary \ref{fome1}). 

Following Theorem~\ref{Hardy-Grafico}, we prove weighted Sobolev embeddings for both subcritical (Theorem~\ref{Sobolev}) and critical (Theorem~\ref{Sobolev p=N}) cases, in addition to a trace embedding theorem (Theorem~\ref{trace-positive}).

 \subsection{Weighted Hardy inequality and consequences: decreasing weight case}

Next, we shall focus on the case of decreasing weights. In addition to $(W_0)$,
we further assume that:
\begin{equation}\label{W_1^-}\tag{$W_1^-$}
\frac{W^p}{[-W_{x_N}]^{p-1}} \in L^1_{\mathrm{loc}}(\Omega) ,\quad W_{x_N} < 0 
\text{ a.e. in } \Omega.\end{equation} 

In this context, our Hardy-type result takes the following form: 
	\begin{theorem}[Hardy inequality of type II]\label{Hardy-Type II} Suppose  \eqref{W_0} and \eqref{W_1^-}. If $p>1$ then, for any $u \in C_0^\infty(\mathbb{R}^N)$, 
    the following inequality holds:
    \begin{equation}\label{Hardy-Geral-Grafico-II}
         \int_{\Omega} [-W_{x_N}]|u|^pdx \leq p^p\int_{\Omega} \frac{W^p}
         {\left[-W_{x_N}\right]^{p-1}}|\nabla u|^pdx +  p	\int_{\mathbb{R}^{N-1}} 
         W(x',\psi(x'))|u(x',\psi(x'))|^pdx'.
  \end{equation}
		\end{theorem}

\begin{remark}
    One can modify the proof of Theorem~\ref{Hardy-Type II} to obtain 
    other interesting results. 
    For instance, if $W$ depends only on $x_N$ 
    and  $W_{x_Nx_N}<0$, then
	$$
	-\dint \Delta W|u|^pdx   \leq p^p\dint \frac{|\nabla W|^{p}}
    {|\Delta W|^{p-1}}|\nabla u|^pdx +p	\int_{\mathbb{R}^{N-1}} 
    W_{x_N}(x',\psi(x'))|u(x',\psi(x'))|^pdx'.
$$
In particular, if $W_{x_N}\leq0$, then 
$$
	\dint |\Delta W||u|^pdx
	\leq p^p\dint \frac{|\nabla w|^{p}}{|\Delta w|^{p-1}}|\nabla u|^pdx,
$$
which is in accordance with the result stated in \cite[Theorem 1]{Davies-Hinz} for bounded domains. 	
\end{remark}
        
\subsection{Comparison with previous results}

Inequalities with non-singular weight were considered in \cite{Ghoussoub-Moradifam}, where the 
notion of \emph{Bessel pairs} was introduced to derive a large class of weighted Hardy-type 
inequalities of the form~\eqref{Geral}. More specifically, for \( a, b > 0 \),
\[
\int_{\mathbb{R}^N} (a + b |x|^{\alpha})^{ - 2/\alpha+\beta} u^2 \, dx
\leq 
C \int_{\mathbb{R}^N} (a + b |x|^{\alpha})^{\beta} |\nabla u|^2 \, dx,\qquad \forall\, u \in C_0^\infty(\mathbb{R}^N),
\]
where the exponents \( \alpha, \beta \) satisfy mild assumptions depending on the dimension. Many 
others have extended this technique of using Bessel pairs to derive Hardy-type 
inequalities. 
For related results, see also \cite{MR2481073,Cossetti-DArca} and 
references therein.

We also mention the works~\cite{MR860920,Lewis,Pfluger}, where the authors proved a non-singular 
Hardy–Sobolev-type inequality involving a boundary trace term. Specifically, for all \( 1 < p < N \), 
there exist constants \( C_1, C_2 > 0 \) such that  
\[
\int_\Omega \frac{|u|^p}{(1 + |x|)^p} \, dx
\leq 
C_1 \int_\Omega |\nabla u|^p \, dx
+ 
C_2 \int_{\partial\Omega} \frac{|\nu \cdot x|}{(1 + |x|)^p} |u| \, dS_x,
\qquad \forall\, u \in C_0^\infty(\mathbb{R}^N),
\]
where \( \Omega \subset \mathbb{R}^N \) is an unbounded domain with a noncompact, smooth 
boundary \( \partial\Omega \), and \( \nu \) denotes the unit outward normal vector 
on \( \partial\Omega \). 
These inequalities make it possible to study elliptic boundary value problems on unbounded domains, 
even when the boundary itself is unbounded, using variational methods.

\subsection{Organization of the paper} 
The rest of the paper is organized as follows. Section~\ref{subsecao-exemplo-1} 
presents several examples that illustrate 
the consequences of our main results. In Section~\ref{sec:inc_weight_proofs}, we provide 
the proofs of the main results 
in the increasing case, whereas Section~\ref{sec:decreasing-case} is devoted to the 
decreasing case. Finally, Section~\ref{Final comments} 
contains some comments on applications of our results and further research projects.

\section{Some examples of weights} \label{subsecao-exemplo-1}
In this section, we discuss some consequences of our results by considering specific 
classes of potentials \(W\). Since \(\Omega\) is diffeomorphic to \(\mathbb{R}^N_+\), 
we restrict our attention to the case \(\Omega = \mathbb{R}^N_+\). Thus, we pick $\psi\equiv 0$ in Theorem \ref{Hardy-Grafico} to get
		\begin{equation}\label{timao-campeao}
			\int_{\mathbb{R}^N_+} W_{x_N}|u|^pdx+p	\int_{\mathbb{R}^{N-1}}  
            W(x',0)|u|^{p}dx'\leq p^p\int_{\mathbb{R}^{N}_+} 
            \frac{W^p}{[W_{x_N}]^{p-1}} |\nabla u|^pdx, \qquad \forall \, u \in C_0^\infty(\mathbb{R}^N).
		\end{equation}

\subsection{Increasing weights}

In what follows, we present examples of weights to which we can apply the last inequality above.

\begin{example} 
     If we consider  $\gamma>0$, $\beta<N-1$ and set
$$
W(x',x_N) := \frac{x_N^\gamma}{|x'|^\beta}, \qquad x'\in\mathbb{R}^{N-1},\,x_N>0,
$$
we have that
$$
W_{x_N}(x) = \frac{\gamma x_N^{\gamma-1}}{|x'|^\beta},\qquad \frac{W(x)^p}{[W_{x_N}(x)]^{p-1}} 
= \frac{1}{\gamma^{p-1}} \frac{x_N^{\gamma+p-1}}{|x'|^\beta}.
$$
Consequently, from \eqref{timao-campeao}, we derive the following inequality:
$$
\int_{\mathbb{R}^N_+}  \frac{1}{x_N^{1-\gamma}|x'|^\beta} |u|^p dx \leq 
\left( \frac{p}{\gamma}\right)^{p} \int_{\mathbb{R}^{N}_+} \frac{x_N^{\gamma+p-1}}{|x'|^\beta} 
|\nabla u|^pdx.
$$
If $0<\gamma<1$, the function on the left-hand side above can blow up along the 
entire hyperplane $\{x_N = 0\}$, in contrast to the usual Hardy inequalities, 
where the singularity occurs only at a point.
\end{example}

\begin{example}\label{ex21} 
    Picking $\gamma>0$, $\beta < N-1$ and setting
$$
W(x',x_N) := \frac{(1+x_N)^\gamma}{|x'|^\beta}, \qquad x'\in\mathbb{R}^{N-1},\,x_N>0,
$$
we have that
$$
W_{x_N}(x) = \frac{\gamma(1+x_N)^{\gamma-1}}{|x'|^\beta},\qquad \frac{W(x)^p}{[W_{x_N}(x)]^{p-1}} = \frac{1}{\gamma^{p-1}} \frac{(1+x_N)^{\gamma+p-1}}{|x'|^\beta}.
$$
A  direct application of   \eqref{timao-campeao} yields
$$
\int_{\mathbb{R}^{N-1}}  \frac{|u|^p}{|x'|^\beta} dx'\leq \left( \frac{p}{\gamma}\right)^{p-1} \int_{\mathbb{R}^{N}_+} \frac{(1+x_N)^{\gamma+p-1}}{|x'|^\beta} |\nabla u|^pdx
$$
after discarding the first term,
which is a Kato-type inequality; see, for instance, \cite{Davila-Dupaigne-Montenegro} for related results.

Moreover, if we consider $\beta=0$, and keep all the terms in~\eqref{timao-campeao}, 
we obtain
		$$
			\begin{aligned}
				\left(\frac{\gamma}{p}\right)^{p}\int_{\mathbb{R}^N_+} \frac{|u|^p}{(1+x_N)^{1-\gamma}}dx&+\left(\frac{\gamma}{p}\right)^{p-1}	\int_{\mathbb{R}^{N-1}} |u|^{p}dx'
				&\leq \int_{\mathbb{R}^{N}_+}(1+x_N)^{\gamma+p-1}|\nabla u|^pdx,
			\end{aligned}
		$$
    and therefore we have obtained a new proof of \cite[Theorem 1.1]{Abreu-Furtado-Medeiros}. 
\end{example}

 \begin{example}  
    If $W(x):=\log(\mathrm{e}+x_N)$, then 
		$$
		W_{x_N}(x)=\frac{1}{e+x_N},\qquad \frac{W(x)^p}{[W_{x_N}(x)]^{p-1}}=(e+x_N)^{p-1} \log^p(e+x_N)
        $$
		and therefore \eqref{timao-campeao} becomes
		\begin{align*}
			\int_{\mathbb{R}^N_+} \frac{|u|^p}{(e+x_N)}dx+ p \int_{\mathbb{R}^{N-1}} |u|^{p}dx'&\leq p^p \int_{\mathbb{R}^N_+} (e+x_N)^{p-1} \log^p(e+x_N)|\nabla u|^pdx,
		\end{align*}
		which is a counterpart of the last inequality obtained in Example \ref{ex21}.
\end{example}
    
 \begin{example}
      Let $\gamma>0$ and consider
		$W(x):=e^{\gamma x_N}$. Then
		$$
		W_{x_N}(x)=\gamma e^{\gamma x_N}, \qquad \frac{W(x)^p}{[W_{x_N}(x)]^{p-1}} =\frac{1}{\gamma^{p-1}}e^{\gamma x_N}
		$$
		and therefore \eqref{timao-campeao} becomes
		$$
			\gamma\int_{\mathbb{R}^N_+} e^{\gamma x_N}|u|^pdx+p\int_{\mathbb{R}^{N-1}} |u|^{p}dx'\leq \gamma \left( \frac{p}{\gamma}\right)^p \int_{\mathbb{R}^{N}_+}e^{\gamma x_N}|\nabla u|^pdx.
		$$
 \end{example}

\begin{example}
    If $W(x):=\arctan(x_N)$, then
		$$
			W_{x_N}(x)=\frac{1}{1+x_N^2},\qquad \frac{W(x)^p}{[W_{x_N}(x)]^{p-1}}= \frac{\arctan^p (x)}{(1+x_N^2)^{p-1}}
		$$
		and therefore \eqref{timao-campeao} gives
		$$
			\int_{\mathbb{R}^N_+} \frac{|u|^p}{(1+x_N^2)}dx\leq \left(\frac{p\pi}{2}\right)^p\int_{\mathbb{R}^{N}_+}(1+x_N^2)^{p-1}|\nabla u|^pdx.
		$$
\end{example}

 \subsection{Decreasing weights}

In what follows, we present examples of weights to which we can apply the Theorem \ref{Hardy-Type II}.

	\begin{example}\label{Example-Decreassing}
        Let  $p>1$, $\gamma<p-1$ and consider
		$W(x):=(1+x_N)^{\gamma-p+1}$. Then,
		$$
		-W_{x_N}(x)=(-\gamma+p-1)(1+x_N)^{\gamma-p}, \qquad
		\frac{W(x)^p}{[-W_{x_N}(x)]^{p-1}}=\frac{(1+x_N)^\gamma}{(-\gamma+p-1)^{p-1}},
		$$
		and therefore Theorem \ref{Hardy-Type II} with $\psi\equiv 0$ yields 
      	$$
			\begin{aligned}
				\left(\frac{-\gamma+p-1}{p}\right)^{p}\int_{\mathbb{R}^N_+} \frac{|u|^p}{(1+x_N)^{p-\gamma}}dx&\leq \int_{\mathbb{R}^{N}_+}(1+x_N)^\gamma|\nabla u|^pdx
                \\
				&+ \left(\frac{-\gamma+p-1}{p}\right)^{p-1}	\int_{\mathbb{R}^{N-1}} |u|^{p}dx'.
			\end{aligned}
		$$
    In particular, if $\gamma=0$, we have the inequality 
        $$
		\int_{\mathbb{R}^N_+} \frac{|u|^p}{(1+x_N)^{p}}dx\leq \left(\frac{p}{p-1}\right)^p\int_{\mathbb{R}^N_+} |\nabla u|^pdx + \left( \frac{p}{p-1} \right) \int_{\mathbb{R}^{N-1}} |u|^{p}dx',
		$$
        which improves \cite[Theorem~1.1]{Abreu-Felix-Medeiros}. 
   \end{example}

   \begin{example}  Let $\gamma>0$ and consider
		$W(s):=e^{-\gamma x_N}$. Then
		$$
		-W_{x_N}(x)=\gamma e^{-\gamma x_N}, \qquad \frac{W(x)^p}{[-W_{x_N}(x)]^{p-1}} =\frac{1}{(-\gamma)^{p-1}}e^{-\gamma x_N}.
		$$
		and therefore
		$$
			\gamma\int_{\mathbb{R}^N_+} e^{-\gamma x_N}|u|^pdx \leq p\int_{\mathbb{R}^{N-1}} |u|^{p}dx' + \frac{p^p}{(-\gamma)^{p-1}}\int_{\mathbb{R}^{N}_+}e^{-\gamma x_N}|\nabla u|^pdx.
		$$
\end{example}

\section{Proofs and Consequences in the Increasing Case} \label{sec:inc_weight_proofs}

This section presents the proofs of Theorem~\ref{Hardy-Grafico}, that is, the Hardy inequality for increasing weight functions $W$, together with several consequences, including weighted Sobolev embeddings.

\begin{proof}[Proof of Theorem~\ref{Hardy-Grafico}]
		Given $u\in C_0^\infty(\mathbb{R}^N)$ and $x'\in \mathbb{R}^{N-1}$, we have that 
\begin{equation*}
     \int_{\psi(x')}^\infty \left(W|u|^p\right)_{x_N}dx_N=-	W(x',\psi(x'))|u(x',\psi(x')|^{p}
 \end{equation*}
 and therefore we can use $W_{x_N}\in L^1_{loc}({\Omega})$ to get
 $$
			\int_{\psi(x')}^\infty W_{x_N}|u|^pdx_N+\int_{\psi(x')}^\infty W(|u|^p)_{x_N}dx_N
            =-	W(x',\psi(x'))|u(x',\psi(x')|^{p}.
		$$
Integrating over $\mathbb{R}^{N-1}$ and using Fubini's Theorem, we obtain
		\begin{align*}
		    - \int_{\mathbb{R}^{N-1}} W(x',\psi(x'))|u(x',\psi(x'))|^{p}dx' & =          
            \int_\Omega W_{x_N}|u|^pdx + \int_\Omega W(|u|^p)_{x_N}dx  \\
            & = \int_\Omega W_{x_N}|u|^pdx + \int_\Omega pW|u|^{p-2} u 
            \frac{\partial u}{\partial x_N}dx.
		\end{align*}
        Hence
		\begin{equation}\label{Rita}
			\dint W_{x_N}|u|^pdx+ 	\int_{\mathbb{R}^{N-1}} W(x',\psi(x'))
            |u(x',\psi(x'))|^{p}dx' 
            \leq p \int_\Omega W|u|^{p-1}|\nabla u|dx.
		\end{equation}
  
  Recall that, for any $\varepsilon>0$ and $a,\,b \geq 0$, Young's inequality yields
		$$
			ab=\left(\frac{\varepsilon}{p-1}\right)^{(p-1)/p}a
            \left(\frac{p-1}{\varepsilon}\right)^{(p-1)/p}b
			 \leq\frac{\varepsilon}{p}a^{p/(p-1)}+\frac{1}{p}
             \left(\frac{p-1}{\varepsilon}\right)^{p-1}b^p.
		$$
Now, taking into account that $W_{x_N}>0$, we have
$$
    W|u|^{p-1}|\nabla u|= \Big( \left[W_{x_N}\right]^{(p-1)/p}|u|^{p-1} \Big) 
    \Big( W \left[W_{x_N}\right]^{{-(p-1)/p}} |\nabla u|\Big).
$$
Since  $W^p/[W_{x_N}]^{p-1}  \in L^1_{loc}(\Omega)$, we can pick $\varepsilon = (p-1)/p$ 
and use the above expression to obtain
		\begin{align*}
		    p\int_\Omega W|u|^{p-1}|\nabla u|dx & \leq \varepsilon\int_\Omega W_{x_N}|u|^pdx
            +\left(\frac{p-1}{\varepsilon}\right)^{p-1}\int_\Omega
            \frac{W^p}{\left[W_{x_N}\right]^{p-1}}|\nabla u|^p dx \\
            & = \frac{p-1}{p} \int_\Omega W_{x_N}|u|^pdx + p^{p-1} \int_\Omega
            \frac{W^p}{\left[ W_{x_N}\right]^{p-1}}|\nabla u|^p dx.
		\end{align*}
Replacing the above  inequality in \eqref{Rita} we obtain the desired result.
\end{proof}

Next, we highlight our inequality in the case of the upper half-space $\mathbb{R}^{N}_+$, 
which can be obtained by setting $\psi\equiv0$.

\begin{corollary}\label{Hardy-Type I} Suppose \eqref{W_0} and \eqref{W_1^+}. Then, for 
any $u \in C_0^\infty(\mathbb{R}^N)$, the following inequality is satisfied:
		\begin{equation}\label{Hardy-Geral-Positiva}
			\int_{\mathbb{R}^N_+} W_{x_N}|u|^pdx+p	\int_{\mathbb{R}^{N-1}}  
            W(x',0)|u|^{p}dx'\leq p^p\int_{\mathbb{R}^{N}_+} 
            \frac{W^p}{[W_{x_N}]^{p-1}} |\nabla u|^pdx.
		\end{equation}
		In particular, if $u \in C_0^\infty(\mathbb{R}^N_+)$, then
		$$
		\int_{\mathbb{R}^N_+} W_{x_N}|u|^pdx\leq p^p\int_{\mathbb{R}^{N}_+} 
        \frac{W^p}{[W_{x_N}]^{p-1}} |\nabla u|^pdx.
		$$
	\end{corollary}

As a consequence, we have the following:

\begin{corollary}\label{Singular-1}
 For any $\gamma>p-1$ the following inequality holds
\begin{equation}\label{Hardy-singular}
 \left|\frac{\gamma -p+1}{p}\right|^{p}\int_{\mathbb{R}^N_+} \frac{|u|^p}{x_N^{p-\gamma}}\,dx \leq  
\int_{\mathbb{R}^N_+} x_N^{\gamma}|\nabla u|^p \, dx,\qquad \forall\, u\in C_0^\infty(\mathbb{R}^N).   \end{equation}
\end{corollary}
\begin{proof}
For any $\varepsilon>0$, we consider the weight $W_\varepsilon(x):=
(\varepsilon+x_N)^\theta$, with $\theta>0$, and use Corollary~\ref{Hardy-Type I} to get
$$
\theta\int_{\mathbb{R}^N_+}(\varepsilon+x_N)^{\theta-1}|u|^pdx+p \varepsilon^\theta
\int_{\mathbb{R}^{N-1}}|u(x',0)|^pdx'\leq\frac{p^p}{\theta^{p-1}}
\int_{\mathbb{R}^N_+}(\varepsilon+x_N)^{p+\theta-1}|\nabla u|^pdx.
$$
Picking  $\theta = \gamma-p+1>0$ and letting $\varepsilon\rightarrow0^+$ we 
obtain \eqref{Hardy-singular}. 
\end{proof}
    
\begin{remark}\label{Remak-Singular}

Corollary~\ref{Singular-1} shows that, if \( \gamma >-1\),  
inequality \eqref{Serinha} holds in  $C_0^\infty(\mathbb{R}^N)$. Moreover, the 
interval $(p-1,+\infty)$  is optimal in the sense that inequality \eqref{Serinha}is 
false if  $\gamma\leq p-1$. In fact, 
if $\gamma\leq p-1$, a function 
$\phi \in C_0^\infty(\mathbb{R}^N)$ such that $\phi \equiv1$ in the 
unit ball $B_1(0)$ provides a counter-example, which one can 
see from the following computation: 
$$
+\infty=\left|\frac{\gamma-p+1}{p}\right|^p\int_{B_1(0) \cap \mathbb{R}^N_+} 
\frac{1}{x_N^{p-\gamma}}\,dx\leq \left|\frac{\gamma-p+1}{p}\right|^p \int_{\mathbb{R}^N_+} 
\frac{|\phi|^p}{x_N^{p-\gamma}}\,dx \leq \int_{\mathbb{R}^N_+} x_N^{\gamma}
|\nabla \phi|^pdx<\infty,
$$
which is a contradiction.

In particular, this shows that the validity of the classical 
Hardy inequality in the upper half-space, as stated in \eqref{Hardy-Classical} 
(which corresponds to $\gamma=0$), does not hold for functions that 
do not vanish along $\{ x_N = 0\}$. Similarly, the 
Hardy-Sobolev-Maz'ya inequality \eqref{Mazya}  does not hold in $C_0^\infty
(\mathbb{R}^N)$. This highlights how the situation changes when we consider functions 
in \( C_0^\infty(\mathbb{R}^N) \) instead of \( C_0^\infty(\mathbb{R}^N_{+}) \) in 
the context of Hardy-type inequalities.

We also emphasise that, in view of  inequality \eqref{Hardy-singular} and a scaling 
argument, we conjecture that the following weighted Hardy-Sobolev-Maz'ya inequality
$$
C_0\left(\int_{\mathbb{R}^N_+} x_N^\frac{N\gamma}{N-p}|u|^{p^*}dx
\right)^{p/p^*}\leq\int_{\mathbb{R}^N_+} x_N^{\gamma}|\nabla u|^p \, dx
-\frac{1}{(\gamma -p+1)^{p}} \int_{\mathbb{R}^N_+} \frac{|u|^p}{x_N^{p-\gamma}}\,dx
,\qquad \forall\, u\in C_0^\infty(\mathbb{R}^N),
$$
holds for any $\gamma>p-1$ and some constant $C_0=C_0(N,p,\gamma)$, where $p^*:=Np/(N-p)$.
\end{remark}

\subsection{Weighted Sobolev embeddings: the increasing case}

To establish the upcoming results, let $\mathcal{D}_W^+$ denote the completion of $C_0^\infty(\mathbb{R}^N)$ with respect to the norm $\|\cdot\|_{\mathcal{D}_W^+}$ given by 
$$
    \|u\|_{\mathcal{D}_W^+}:=\left(\dint \frac{W^p}{[W_{x_N}]^{p-1}} 
    |\nabla u|^pdx\right)^{1/p}.
$$

We interpolate the weighted Hardy inequality \eqref{Hardy-Geral-Grafico} with classical Sobolev inequalities to derive weighted Sobolev embeddings. Such embeddings play a central role in the analysis of nonlinear elliptic equations. We first introduce some notation. For all $q \geq 1$ and non-negative $V \in L^1_{loc}(\Omega)$, define the weighted Lebesgue space 
	$$
		L^q(\Omega,V) := \left\{ u \in L^1_{loc}(\Omega)\,:\, \|u\|_{q,V} := 
        \left( \dint V(x)|u|^q dx \right)^{1/q} < +\infty \right\}.
	$$
From now on, we assume that $\psi \in C^1(\mathbb{R}^{N-1}, \mathbb{R})$, since we shall make use of the change of variables theorems.   If we denote, for $s \in [0,p]$, 
$$
q(s) := \frac{(N-s)p}{(N-p)}, 
$$
and assume
   \begin{equation} \label{W_2^+} \tag{$W_2^+$}
     \text{there exists } c_1>0 \text{ such that } \frac{W^p}{[W_{x_N}]^{p-1}} 
     \geq c_1 \text { a.e. in } \Omega.
    \end{equation} 
we shall prove the following:
    \begin{theorem}[Sobolev embedding]\label{Sobolev} 
Suppose \eqref{W_0}, \eqref{W_1^+} and   \eqref{W_2^+}. If $1<p<N$, then the following 
weighted Sobolev embedding is continuous
  $$
  \mathcal{D}_W^+ \hookrightarrow L^{q(s)}(\Omega,W^{s/p}_{x_N}),  \quad 
  \text{for all } s\in [0,p].
  $$
	\end{theorem}
 	
The 
starting pointing is noticing  that the map $\Phi:\Omega \to \mathbb{R}^N_+$ given by  
    $$
    \Phi(x',x_N) := (x',x_N-\psi(x'))
    $$
    is a diffeomorphism. So, given $u \in C_0^\infty(\mathbb{R}^N)$, we can use the change 
    of variable $x=\Phi^{-1}(y)$  and 
    the classical Gagliardo–Nirenberg–Sobolev inequality, hereafter referred to 
    as the GNS inequality,
     for $v:= u \circ \Phi^{-1}$, to get
    \begin{align*}
   \dint |u(x)|^{p^*} dx &= \int_{\mathbb{R}^N_+} |u(\Phi^{-1}(y))|^{p^*} 
   \left| J_{\Phi^{-1}}(y)\right| dy \\ & = \int_{\mathbb{R}^N_+} |v(y)|^{p^*} dy 
   \leq C \left( \int_{\mathbb{R}^N_+} |\nabla v(y)|^p dy \right)^{p^*/p}.
\end{align*}
    Using now the change of variables $y=\Phi(x)$ in the last integral above, we obtain
  \begin{equation}\label{GNS-Grafico}
    \dint |u|^{p^*} dx \leq C \left( \dint |\nabla u|^p dx \right)^{p^*/p}, \qquad 
    \forall\,u \in C_0^1(\mathbb{R}^N),
    \end{equation}
    that is, the GNS inequality holds with the integrals taken over the set $\Omega$.\\

We prove now our first embedding result.
\begin{proof}[Proof of Theorem \ref{Sobolev}]
We first recall that
$$q(s)=\frac{p(N-s)}{N-p}.
$$
If $s=p$, then $q(p)=p$ and the result is a direct consequence of \eqref{Hardy-Geral-Grafico}. On the other hand, since $q(0)=Np/(N-p)=p^*$,  
it follows from the GNS inequality \eqref{GNS-Grafico} and \eqref{W_2^+} that
 \begin{align*}
     \|u\|_{p^*}^{p^*} = \dint  |u|^{p^*}dx
			& \leq C\left(\dint  |\nabla u|^pdx\right)^{{p^*}/{p}} \\
            & \leq C \left( \frac{1}{c_1} \dint \frac{W^p}{[W_{x_N}]^{p-1}} 
            |\nabla u|^p dx \right)^{p^*/p} = 
            C c_1^{-p^*/p} 
            \|u\|_{\mathcal{D}_W^+}^{p^*},
 \end{align*}
 which proves the case $s=0$. 

For $0<s<p$, using H\"older's inequality, Theorem~\ref{Hardy-Grafico} and the 
above inequality  we have 
\begin{align*}
\int_\Omega W_{x_N}^{s/p}|u|^{q(s)}dx&=\int_\Omega (W_{x_N}^{s/p}|u|^s)|u|^{q(s)-s}dx\\
&\leq\left(\int_\Omega W_{x_N}|u|^pdx\right)^{s/p}\left(\int_\Omega|u|^{p^*}dx
\right)^{(p-s)/p}. \\
& \leq \left( p^p \|u\|_{\mathcal{D}_W^+}^p\right)^{s/p}  \left( 
Cc_1^{-p^*/p}\|u\|_{\mathcal{D}_W^+}^{p^*}\right)^{(p-s)/p}  = 
C_1 \|u\|_{\mathcal{D}_W^+}^{s + p^*(p-s)/p}.
\end{align*}
Using the definition of $q(s)$ we can show that
$$
s + \frac{p^*(p-s)}{p} = q(s)
$$
and we have done.
\end{proof}

\begin{remark}\label{Remark-General}Under the same hypotheses of Theorem \ref{Sobolev}, 
suppose also $W_{x_N} \in L^\infty({\Omega})$. Given $q_0 \in [p,p^*]$, we can use 
the definition of $q(s)$ 
to obtain $s_0 \in [0,p]$ such that $q(s_0)=q_0$. Thus
\begin{align*}
\int_\Omega W_{x_N}|u|^{q_0} dx &= \int_\Omega W_{x_N}^{(p-s_0)/p}W_{x_N}^{s_0/p}|u|^{q(s)} dx  
\\& \leq \|W_{x_N}\|_{L^{\infty}(\Omega)}^{(p-s)/p} \int_\Omega W_{x_N}^{s_0/p}|u|^{q(s_0)} dx 
\leq C_2 \|u\|_{\mathcal{D}_W^{+}}^{q(s_0)}
\end{align*}
and hence the Sobolev embedding $\mathcal{D}_W^+ \hookrightarrow L^{q_0}(\Omega,W_{x_N})$ is 
continuous, for all $q_0\in[p,p^*]$.
\end{remark}

    By using an iterative process based on the Gagliardo-Nirenberg-Sobolev inequality, 
    we also derive an embedding result when $p=N$. In this case,  we require that $W$ 
    depends only on $x_N$ and has slightly more regularity. Our result in the case $p=N$ is:
	
	\begin{theorem}[Sobolev embedding, borderline case]\label{Sobolev p=N} Suppose 
     $W=W(x_N)\in C^1((0,\infty),\mathbb{R})$ satisfies \eqref{W_0}, \eqref{W_1^+} 
     and \eqref{W_2^+}. If $p=N$, then the following weighted Sobolev embedding is continuous
    $$
    \mathcal{D}_W^+ \hookrightarrow L^q(\Omega,W_{x_N}),  \quad \text{for all } 
    q \in [N,+\infty).
    $$
    \end{theorem}

    	\begin{proof}
		It follows from  \eqref{GNS-Grafico},  with $p=1$, that 
		\begin{equation}\label{GNS n=1}
			\left(\dint |v|^{\frac{N}{N-1}}dx\right)^{\frac{N-1}{N}} \leq C_1\dint |\nabla v|dx,
            \quad \forall \,v \in C_0^1(\mathbb{R}^N).
		\end{equation}
		 Given $u \in C^{\infty}_0(\mathbb{R}^N)$, we may recall that $W \in C^1$ to  
         pick   $v:=W|u|^N \in C_0^1(\mathbb{R}^N)$  in the above inequality  and obtain
      $$
		\begin{aligned}
			\left(\dint W^{\frac{N}{N-1}}(x_N)|u|^{\frac{N^2}{N-1}}dx\right)^{\frac{N-1}{N}}& 
            \leq C_1\dint |\nabla (W(x_N)u^N)|dx\\
			&= C_1  \dint \left[ W_{_{x_N}}(x_N)|u|^N+NW(x_N)|u|^{N-1}|\nabla u|\right]dx.
		\end{aligned}
		$$
        From Theorem \ref{Hardy-Grafico}, we obtain 
		$$
		\dint W_{x_N}(x_N)|u|^Ndx\leq N^ N\|u\|_{\mathcal{D}_W^+}^N.
		$$
		On the other hand, since
        $$
        W|u|^{N-1}|\nabla u|= \Big( [W_{x_N}]^{(N-1)/N}|u|^{N-1} \Big) 
        \Big( W [W_{x_N}]^{{-(N-1)/N}} |\nabla u|\Big),
        $$
        we can apply Young's inequality with 
        exponents $N/(N-1)$ and $N$ to get
      \begin{align*}
			N\dint W(x_N)|u|^{N-1}|\nabla u|dx	\leq&(N-1)
            \dint W_{x_N}(x_N)|u|^Ndx+\|u\|^N_{\mathcal{D}_W^+}.
		\end{align*}
        Using \eqref{W_2^+} and all the above inequalities, we obtain
        \begin{equation}\label{key2}	
        \|u\|_{\frac{N^2}{N-1},W_{x_N}}^N   \leq \frac{1}{c_1}\left(\dint  
        W(x_N)^{\frac{N}{N-1}}|u|^{\frac{N^2}{N-1}}dx\right)^{\frac{N-1}{N}} \leq 
        C_2\|u\|_{\mathcal{D}_W^+}^N,
        \end{equation}
        with $C_2=C_1 (N^{N+1}+1)$ depending only on $N$. 
        
        We infer from the above inequality that the statement of the theorem holds 
        for $q=N^2/(N-1)$.  
        On the other hand, Theorem \ref{Hardy-Grafico} implies that it also holds 
        for $q=N$. Hence, if $N<q<N^2/(N-1)$, we can consider $\theta \in (0,1)$ such that
        $$
            \frac{1}{q} = (1-\theta)\frac{1}{N} + \theta \frac{(N-1)}{N^2}
        $$
        and use H\"older's inequality with exponents $s=\frac{N^2}{(N-1)\theta q}$ 
        and $s'=\frac{N}{(1-\theta)q}$ to get
        \begin{align*}
        \dint W_{x_N}(x_N)|u|^q dx &= \dint [W_{x_N}(x_N)]^{1/s} |u|^{\theta q} 
        [W_{x_N}(x_N)]^{1/s'}|u|^{(1-\theta)q} dx \\ & \leq \left( 
        \dint W_{x_N}(x_N)|u|^{\frac{N^2}{N-1}} dx \right)^{\frac{(N-1)}{N^2} 
        \theta q} \left( \dint W_{x_N}(x_N)|u|^{N} dx \right)^{\frac{1}{N}(1-\theta)q} \\
        & \leq \|u\|_{\mathcal{D}_W^+}^{\theta q} \|u\|_{\mathcal{D}_W^+}^{(1-\theta) q}  
        = \|u\|_{\mathcal{D}_W^+}^{q},
        \end{align*}
        and therefore  
        Theorem \ref{Sobolev p=N} is valid for all $q \in [N,N^2/(N-1)]$. 
        
        We now notice that, since $N<N+1<N^2/(N-1)$,  we have that
		$$
			\left(\dint W_{x_N}(x_N)|u|^{N+1}dx\right)^{\frac{N}{N+1}}\leq 
            C_3\|u\|_{\mathcal{D}_W^+}^N.
		$$
		Hence, applying~\eqref{GNS n=1} with $v:=W|u|^{N+1}$ and arguing as before, we get
			\begin{align*}
			\left(\dint W^{\frac{N}{N-1}}(x_N)|u|^{\frac{N(N+1)}{(N-1)}} \right)^{\frac{N-1}{N}}
            &\leq C_4\left(\dint W_{x_N}(x_N)|u|^{N+1}dx +\dint W(x_N)|u|^N
            |\nabla u|dx\right) \\
			   &\leq  C_4 \left( C_3^{\frac{N+1}{N}}\|u\|_{\mathcal{D}_W^+}^{N+1} + 
               \dint W(x_N)|u|^N|\nabla u|dx\right) .
		\end{align*}
		Using H\"older's inequality with exponents $N/(N-1)$ and $N$, 
        besides \eqref{key2}, we obtain
      \begin{align*}
			\dint  W(x_n)|u|^{N}|\nabla u|dx=&\dint \left( 
            [W_{x_N}(x_N)]^{\frac{N-1}{N}}|u|^N \right) 
            \left( \frac{W(x_N)}{[W_{x_N}(x_N)]^{\frac{N-1}{N}}}|\nabla u| \right)dx\\
			\leq&\left(\dint W_{x_N}(x_N)|u|^{\frac{N^2}{(N-1)}}dx\right)^{\frac{N-1}{N}} 
            \|u\|_{\mathcal{D}_W^+}  \leq C_2\|u\|_{\mathcal{D}_W^+}^{N}.
		\end{align*}
        All the above inequalities, together with \eqref{W_2^+}, provide
        \begin{align*}
			\|u\|_{\frac{N(N+1)}{(N-1)},W_{x_N}}^N & = \left(\dint 
            W_{x_N}(x_N)|u|^{\frac{N(N+1)}{(N-1)}}dx \right)^{\frac{N-1}{N+1}} \\
            &\leq c_1^{-1/(N+1)}\left(\dint W^{\frac{N}{N-1}}(x_N)|u|^{\frac{N(N+1)}{(N-1)}}dx 
            \right)^{\frac{N-1}{N+1}} \leq C_5 \|u\|_{\mathcal{D}_W^+}^N,
		\end{align*}
        and therefore  the continuous  embedding holds for $N \leq q \leq N(N+1)/(N-1)$.
		
		We can now iterate this process to obtain the validity of the embedding in any interval 
        of type $\left[N,N(N+k)/(N-1)\right]$,  with $k \in \mathbb{N}$, concluding the 
        proof of the theorem.
	\end{proof}

    \begin{remark} \label{chuva-sem-x'}
        It is interesting to look for conditions which ensure a result similar to that 
        of Theorem~\ref{Sobolev p=N} when the potential $W$ depends also on 
        $x' \in \mathbb{R}^{N-1}$. A simple inspection of the proof shows that 
        this is the case if $W=W(x)$ satisfies the following assumption:
\begin{equation} \label{chuva}
    |\nabla W(x)| \leq c_4|W_{x_N}(x)|, \qquad \forall \, x \in \Omega.
\end{equation}
Actually, if this is true, we can compute at the beginning of the proof
$$
\dint |\nabla W(x)||u|^N dx \leq C_6 \dint |W_{x_N}(x)||u|^N dx + 
N \dint W(x)|u|^{N-1}|\nabla u| dx,
$$
and the result follows from the original argument line by line.

An inequality like the one in \eqref{chuva} holds, for example, if the potential $W$ 
has the  form $W(x', x_N) = W_1(x_N)W_2(x')$  with 
        $$
        W_1 \left[ \frac{\partial W_1}{\partial x_N}    \right]^{-1} \in  
        L^\infty((0,\infty)), \qquad \frac{\partial W_2}{\partial x_i} \in 
        L^\infty(\mathbb{R}^{N-1}),  
        $$
        for any $i=1,\ldots,N-1$. 
    \end{remark}

\begin{remark} \label{embedding-in-w1p} Suppose \eqref{W_0}, \eqref{W_1^+}, 
\eqref{W_2^+} and additionally  $W_{x_N}(x) \geq C_1 >0$, for any $x\in \Omega$. Then, 
given $u\in C_0^\infty(\mathbb{R}^N)$, we may invoke  Theorem ~\ref{Hardy-Grafico} to get	
		$$
			C_1 \dint  |u|^pdx 	\leq p^p\dint \frac{W^p}{[W_{x_N}]^{p-1}}|\nabla u|^p dx.
		$$
        Moreover, from \eqref{W_2^+}, we obtain
        $$
        \dint |\nabla u|^p dx \leq \frac{1}{c_1} \dint 
        \frac{W^p}{[W_{x_N}]^{p-1}}|\nabla u|^p dx.
        $$
        So, under this setting, we obtain the continuous embedding 
		$$
		\mathcal{D}_W^+\hookrightarrow W^{1,p}(\Omega).
		$$	
\end{remark}

    In our next result, we are interested in giving a precise meaning to the values of  
    functions in  $\mathcal{D}_W^+$ on the boundary $\partial\Omega$. To this end, we assume 
    that $\psi:\mathbb{R}^{n-1}\to\mathbb{R}$ is a globally Lipschitz continuous function, and 
    we use the standard notation $p_* := p(N-1)/(N-p)$ for the critical Sobolev exponent of 
    the trace embedding.
    
  We shall prove the following result:
\begin{theorem}[Trace Embedding]\label{trace-positive}
		Suppose \eqref{W_0}, \eqref{W_1^+} and  \eqref{W_2^+}. In addition, 
        suppose $W_{x_N} \in L^{\infty}(\Omega)$, if $1<p<N$, and $W \in 
        \mathcal{C}^1((0,+\infty),\mathbb{R})$, if $p=N$.        The following weighted 
        Sobolev trace embedding is continuous        
$$       \mathcal{D}_W^+ \hookrightarrow L^q(\partial\Omega, W) , \quad 
\text{for all } q \in \begin{cases} [p,p_*], & \text{if } 1<p<N; \medskip \\ 
[N,+\infty), & \text{if } p=N.\end{cases}.
$$
 \end{theorem} 

 \begin{proof}
		First we consider $1<p<N$ and $u\in C_0^\infty(\mathbb{R}^N)$. 
        Since $\psi$ is a globally Lipchitz continuous function, there exists 
        a constant $C_1>0$ such that 
        $$
        \sqrt{1+\|\nabla \psi(x')\|^2 } \leq C_1,\qquad \forall \,x'\in \mathbb{R}^{N-1}.
        $$
        Thus, in view of Theorem \ref{Hardy-Grafico} one has
        \begin{equation} \label{JP-1}
                \begin{aligned}
    	\|u\|_{L^p(\partial\Omega,W)}^p & =  \int_{\mathbb{R}^{N-1}} W(x',\psi(x')||u(x',\psi(x'))|^p \sqrt{1+\|\nabla \psi(x')\|^2 }dx' \\& \leq C_1 \int_{\mathbb{R}^{N-1}} W(x',\psi(x')||u(x',\psi(x'))|^p dx'  \\& \leq C_1 p^{p-1}\|u\|_{\mathcal{D}_W^+}^p.
    \end{aligned}
    \end{equation}
        From the above inequality, we conclude that  the trace embedding $\mathcal{D}_W^+ \hookrightarrow L^p(\partial\Omega,W)$ is continuous.
        
        Concerning $q=p_*$, we pick $x' \in \mathbb{R}^{N-1}$ and notice that
        \begin{align*}       W(x',\psi(x'))|u(x',\psi(x'))|^{p_*}=&-\int_{\psi(x')}^{\infty}\left(W|u|^{p_*}\right)_{x_N}dx_N\\		\leq&\int_{\psi(x')}^{\infty}W_{x_N}|u|^{p_*}dx_N+p_*\int_{\psi(x')}^{\infty}W|u|^{p_*-1}|\nabla u|dx_N.
		\end{align*}
		Integrating over $\mathbb{R}^{N-1}$ and using Fubini's Theorem, we get 
		\begin{equation} \label{JP-dez}
		\int_{\mathbb{R}^{N-1}}W(x',\psi(x'))|u(x',\psi(x'))|^{p_*}dx'\leq \dint W_{x_N}|u|^{p_*}dx + p_*\dint W|u|^{p_*-1}|\nabla u|dx.
		\end{equation}
		It follows from   Remark~\ref{Remark-General} that   
		\begin{equation*}		\dint W_{x_N}|u|^{p_*}dx\leq C_2\|u\|^{p_*}_{\mathcal{D}_W^+}.
		\end{equation*}
        Moreover, we can use 
        $$
        W|u|^{p_*-1}|\nabla u|= \Big( [W_{x_N}]^{(p-1)/p}|u|^{p_*-1} \Big) \Big( W [W_{x_N}]^{{-(p-1)/p}} |\nabla u|\Big),
        $$
        Hölder's inequality with exponents $p/(p-1)$ and $p$,  and Theorem \ref{Sobolev} again, to obtain
        \begin{align*}
			\dint W|u|^{p_*-1}|\nabla u|dx & \leq  \left(\dint W_{x_N}|u|^{p^*}dx\right)^{(p-1)/p} \|u\|_{\mathcal{D}_W^+} \\
            & \leq C_2 \|u\|_{\mathcal{D}_W^+}^{p_*-1} \|u\|_{\mathcal{D}_W^+} = C_3 \|u\|_{\mathcal{D}_W^+}^q.
		\end{align*}

        Using the above inequalities and arguing as in the proof of case $q=p$, we get
        $$
			\|u\|_{L^{p_*}(\partial\Omega,W)}^{p_*} \leq C_2\int_{\mathbb{R}^{N-1}} W(x',\psi(x'))|u(x',\psi(x'))|^{p_*} dx'  \leq C_4 \|u\|_{\mathcal{D}_W^+}^{p_*},
		$$
        with $C_4= C_1p^{p-1}(C_2+p_*C_3)$, 
        and therefore the embedding  $\mathcal{D}_W^+ \hookrightarrow L^{p_*}(\partial\Omega)$ is also continuous. The case $q \in (p,p_*)$ can be proved using interpolation,  as done in the proof of Theorem \ref{Sobolev p=N}. We omit the details.

          Suppose now that  $p=N$. Let $ q \geq N$ and $u \in C_0^\infty(\mathbb{R}^N)$. Arguing as in the proof of \eqref{JP-dez},  we obtain
		\begin{equation} \label{JP-2}
		    \begin{aligned}
		 \int_{\mathbb{R}^{N-1}}W(x',\psi(x'))|u(x',\psi(x'))|^qdx' & \leq \dint W_{x_N}|u|^qdx + q \dint W|u|^{q-1}|\nabla u|dx \\
         & \leq C_1\|u\|^q_{\mathcal{D}_W^+} + q \dint W|u|^{q-1}|\nabla u|dx, 
		\end{aligned}
        \end{equation}
        where we have used Theorem \ref{Sobolev p=N} in the last inequality.        Since $(q-1)N/(N-1)\geq N$, we can use 
        Hölder's inequality with exponents $N/(N-1)$ and $N$,  and Theorem \ref{Sobolev} again to obtain
        \begin{align*}
			\dint W|u|^{q-1}|\nabla u|dx & \leq  \left(\dint W_{x_N}|u|^{(q-1)N/(N-1)}dx\right)^{(N-1)/N} \|u\|_{\mathcal{D}_W^+} \\
            & \leq C_2 \|u\|_{\mathcal{D}_W^+}^{q-1} \|u\|_{\mathcal{D}_W^+} = C_2 \|u\|_{\mathcal{D}_W^+}^q
		\end{align*}
        and the result follows from \eqref{JP-2} and expression \eqref{JP-1} with $p$ replaced by $q$. 
        \end{proof}

Let $\psi_1, \psi_2 : \mathbb{R}^{N-1} \to \mathbb{R}$ be continuous functions 
with $\psi_1(x') < \psi_2(x')$, for any $x'\in\mathbb{R}^{N-1}$. A version of 
Theorem~\ref{Hardy-Grafico} also holds when $\Omega=\Omega_1 \cup \Omega_2$ where
\[
    \Omega_1 = \{ x = (x',x_N) : x_N < \psi_1(x')  \} \quad \text{and} \quad 
        \Omega_2 = \{ x = (x',x_N) :  \psi_2(x') < x_N \}.
\]
\begin{corollary} \label{fome1}
Let $\psi_1 <  \psi_2$, let $\Omega_1$ be the subgraph of $\psi_1$ and $\Omega_2$ 
the supergraph of $\psi_2$ (as defined above) and let $\Omega = \Omega_1 \cup
\Omega_2$. Then, for any $u \in C_0^\infty(\mathbb{R}^N)$, the following inequality is satisfied:
\begin{multline}\label{bora}
    \int_{\Omega} W_{x_N} |u|^p \, dx 
    + p \int_{\mathbb{R}^{N-1}} W(x',\psi_1(x'))\, |u(x',\psi_1(x'))|^p \, dx' \\[0.4em]
    \leq 
    p \int_{\mathbb{R}^{N-1}} W(x',\psi_2(x'))\, |u(x',\psi_2(x'))|^p \, dx'
    + p^p \int_{\Omega} \frac{W^p}{[W_{x_N}]^{p-1}}\, |\nabla u|^p \, dx.
\end{multline}
\end{corollary}

\begin{proof} 
By inspecting the proof of Theorem~\ref{Hardy-Grafico}, one observes that
\[
    \int_{\Omega_1} W_{x_N} |u|^p \, dx 
    + p \int_{\mathbb{R}^{N-1}} W(x',\psi_1(x'))\, |u(x',\psi_1(x'))|^p \, dx' 
    \leq p^p \int_{\Omega_1} \frac{W^p}{[W_{x_N}]^{p-1}}\, |\nabla u|^p \, dx,
\]
and
\[
    \int_{\Omega_2} W_{x_N} |u|^p \, dx 
    \leq p \int_{\mathbb{R}^{N-1}} W(x',\psi_2(x'))\, |u(x',\psi_2(x'))|^p \, dx' 
    + p^p \int_{\Omega_2} \frac{W^p}{[W_{x_N}]^{p-1}}\, |\nabla u|^p \, dx.
\]
Adding these two inequalities yields~\eqref{bora}.
\end{proof}

\section{Proofs and Consequences in the decreasing Case} \label{sec:decreasing-case}
This section presents the proofs of Theorem~\ref{Hardy-Type II}, namely the Hardy inequality for decreasing weight functions $W$. We also derive several consequences, including weighted Sobolev embedding results.

\begin{proof}[Proof of Theorem \ref{Hardy-Type II}]
The proof is a straightforward adaptation of Theorem~\ref{Hardy-Grafico}, and we only sketch 
the main steps.   Given  \(u \in C_0^\infty(\mathbb{R}^N)\), we can use the Fundamental 
Theorem of Calculus, Fubini's theorem and \eqref{W_1^-}, to get
\[
\int_\Omega W_{x_N} |u|^p \, dx + \int_\Omega W \, (|u|^p)_{x_N} \, dx
= - \int_{\mathbb{R}^{N-1}} W(x',\psi(x')) |u(x',\psi(x'))|^p \, dx'
\]
and therefore
\[
-\int_\Omega W_{x_N} |u|^p \, dx 
\le p \int_\Omega W |u|^{p-1} |\nabla u| \, dx 
+ \int_{\mathbb{R}^{N-1}} W(x',\psi(x')) |u(x',\psi(x'))|^p \, dx'.
\]

Since \(W_{x_N} < 0\), we may choose $\varepsilon>0$, write
\[
W |u|^{p-1} |\nabla u| 
= \big( [-W_{x_N}]^{\frac{p-1}{p}} |u|^{p-1}\big) 
\big( W [-W_{x_N}]^{-\frac{p-1}{p}} |\nabla u| \big).
\]
and apply Young's inequality to obtain
\[
p \int_\Omega W |u|^{p-1} |\nabla u| \, dx
\le \varepsilon \int_\Omega [-W_{x_N}] |u|^p \, dx
+ \left(\frac{p-1}{\varepsilon}\right)^{p-1} \int_\Omega \frac{W^p}{[-W_{x_N}]^{p-1}} |\nabla u|^p \, dx.
\]
Choosing \(\varepsilon := (p-1)/p\) and substituting into the previous inequality, the result follows.
\end{proof}

\subsection{Weighted Sobolev embeddings: The decreasing case}
Assume that the weight function $W$ is decreasing in the $x_N$-direction. We introduce the space $\mathcal{D}_W^-$ as the completion of $C_0^\infty(\mathbb{R}^N)$ with respect to the norm
\[
\|u\|_{\mathcal{D}_W^-}
:=\left(
\int_\Omega \frac{W^p}{[-W_{x_N}]^{p-1}}\,|\nabla u|^p \, dx
+\int_{\mathbb{R}^{N-1}} W(x',\psi(x'))\,|u(x',\psi(x'))|^p \, dx'
\right)^{1/p}.
\]
We further assume that
\begin{equation} \label{W_2^-} \tag{$W_2^-$}
\text{there exists } c_2>0 \text{ such that } 
\frac{W^p}{[-W_{x_N}]^{p-1}} \geq c_2 \quad \text{a.e. in } \Omega.
\end{equation}
Under this assumption, we establish the following result.

\begin{theorem}\label{Sobolev negative} Suppose \eqref{W_0}, \eqref{W_1^-} and \eqref{W_2^-}. If  $1<p<N$,  
the following weighted Sobolev embedding is continuous
  $$
  \mathcal{D}_W^- \hookrightarrow L^{q(s)}(\Omega,-W_{x_N}^{s/p}),  \quad \text{for all } s \in [0,p].
  $$
\end{theorem}

\begin{proof}
The cases $s=0$ and $s=p$ follows from GNS inequality and Theorem Theorem \ref{Hardy-Type II}, respectively.
For $0<s<p$, using H\"older's inequality we have 
$$
\begin{aligned}
\int_\Omega [-W_{x_N}]^{s/p}|u|^{q(s)}dx & =\int_\Omega ([-W_{x_N}]^{s/p}|u|^s)|u|^{q(s)-s}dx 
 \\ & \leq\left(\int_\Omega [-W_{x_N}]|u|^pdx\right)^{s/p}\left(\int_\Omega|u|^{p^*}dx\right)^{(p-s)/p}.
\end{aligned}
$$
Now it is sufficient to use  Theorem \ref{Hardy-Type II},  the GSN inequality, and condition $(W_2^-)$ to obtain the desired result. 
\end{proof}

\begin{theorem}\label{Sobolev negative-borderline} Suppose $W\in C^1((0,\infty),\mathbb{R})$ depends only on $x_N$ and  satisfies  \eqref{W_1^-},  \eqref{W_2^-}. If  $p=N$,  
the following weighted Sobolev embedding is continuous
 $$
    \mathcal{D}_W^- \hookrightarrow L^q(\Omega,-W_{x_N}),  \quad \text{for all } q \in [N,+\infty).
    $$
\end{theorem}

\begin{proof}
    Given $u \in C^{\infty}_0(\mathbb{R}^N)$, we replace   $v:=W|u|^N \in C_0^1(\mathbb{R}^N)$ in \eqref{GNS n=1} to get 
      $$
			\left(\dint W^{\frac{N}{N-1}}|u|^{\frac{N^2}{N-1}}dx\right)^{\frac{N-1}{N}}\leq C  \dint \left[-W_{{x_N}}|u|^N+NW|u|^{N-1}|\nabla u|\right]dx.
		$$
        Using Theorem \ref{Hardy-Type II} and arguing as in the proof of Theorem \ref{Sobolev p=N}, we obtain 
		$$
			\|u\|_{\frac{N^2}{N-1},-W_{x_N}}^N \leq C_1\left(\dint  W^{\frac{N}{N-1}}|u|^{\frac{N^2}{N-1}}dx\right)^{\frac{N-1}{N}}\leq C_1\|u\|_{\mathcal{D}_W^-}^N,
		$$
		with $C_1 = C_1(N)>0$, and therefore the theorem holds $q=N^2/(N-1)$.  Moreover,   Theorem \ref{Hardy-Type II} implies that the result also holds for $q=N$. Hence, we can use interpolation to fulfil the interval  $q \in [N,N^2/(N-1)]$.

        Since $N<N+1<N^2/(N-1)$, we can repeat the above argument and use the same ideas performed in the proof of Theorem \ref{Sobolev p=N} to conclude that the embedding holds for $N \leq q \leq N(N+1)/(N-1)$. 
		By iterating this process, we obtain the validity of the embedding in any interval of type $\left[N,N(N+k)/(N-1)\right]$,  with $k \in \mathbb{N}$, concluding the proof of the theorem.
\end{proof}

\begin{remark}  
    Following a procedure similar to the one in Remark \ref{Remark-General}, and in addition to the hypotheses of Theorem \ref{Sobolev}, if we assume that $W_{x_N} \in L^\infty({\Omega})$, then the weighted Sobolev embedding $\mathcal{D}_W^- \hookrightarrow L^q(\Omega,-W_{x_N})$ is continuous for all $q \in [p, p^*]$.
\end{remark}

The following results concern the corresponding trace embeddings:

\begin{theorem}\label{trace-negative}
		Suppose \eqref{W_0}, \eqref{W_1^-} and  \eqref{W_2^-}. In addition, suppose $W_{x_N} \in L^{\infty}(\Omega)$, if $1<p<N$, and $W \in C^1((0,+\infty),\mathbb{R})$, if $p=N$.        The following weighted 
        Sobolev trace embedding is continuous        
$$       \mathcal{D}_W^- \hookrightarrow L^q(\partial\Omega, W) , \quad \text{for all } q \in \begin{cases} [p,p_*], & \text{if } 1<p<N; \medskip \\ 
[N,+\infty), & \text{if } p=N.\end{cases}.
$$
 \end{theorem} 

\begin{proof}
For $1<p<N$, we can use the first two lines of \eqref{JP-2} and the definition of $\|\cdot\|_{\mathcal{D}_W^-}$ to conclude that the embedding $\mathcal{D}_W^- \hookrightarrow L^p(\partial\Omega,W)$ is continuous. So,  as before, we need only to prove the continuity of $\mathcal{D}_W^- \hookrightarrow L^{p_*}(\partial\Omega,W)$ to fullfill the interval $q\in [p,p_*]$ via interpolation.

For any $x' \in \mathbb{R}^{N-1}$, we have that
$$
W(x',\psi(x'))|u(x',\psi(x'))|^{p_*}=\int_{\psi(x')}^{\infty}\left(-W|u|^{p_*}\right)_{x_N}dx_N		
$$        
and therefore
\begin{align*}
\int_{\mathbb{R}^{N-1}}W(x',\psi(x'))|u(x',\psi(x'))|^{p_*}dx' & \leq \dint [-W_{x_N}]|u|^{p_*}dx + p_*\dint W|u|^{p_*-1}|\nabla u|dx \\
& \leq C_1\|u\|^{p_*}_{\mathcal{D}_W^-} + p_*\dint W|u|^{p_*-1}|\nabla u|dx,
\end{align*}
where we have used the embedding $\mathcal{D}_W^- \hookrightarrow L^{p_*}(\Omega,-W_{x_N})$. For the second integral in the right-hand side above we use
        $$
        W|u|^{p_*-1}|\nabla u|= \Big( [-W_{x_N}]^{(p-1)/p}|u|^{p_*-1} \Big) \Big( W [-W_{x_N}]^{{-(p-1)/p}} |\nabla u|\Big),
        $$
        Hölder's inequality and the Sobolev embedding again to obtain
        \begin{align*}
			\dint W|u|^{p_*-1}|\nabla u|dx & \leq  C_2 \|u\|_{\mathcal{D}_W^-}^{p_*}.
		\end{align*}
        The above inequalities and the same argument used for $q=p$ yield
        $$
			\|u\|_{L^{p_*}(\partial\Omega,W)}^{p_*} \leq C_3\int_{\mathbb{R}^{N-1}} W(x',\psi(x'))|u(x',\psi(x'))|^{p_*} dx'  \leq C_4 \|u\|_{\mathcal{D}_W^-}^{p_*}.
		$$
        
    For the case $p=N$, we pick $q \geq N$,  $u \in C_0^\infty(\mathbb{R}^N)$ repeat the above argument to get
   	\begin{equation} \label{JP-20}
		 \int_{\mathbb{R}^{N-1}}W(x',\psi(x'))|u(x',\psi(x'))|^qdx' 
          \leq C_5\|u\|^q_{\mathcal{D}_W^+} + q \dint W|u|^{q-1}|\nabla u|dx. 
        \end{equation}
   Since $(q-1)N/(N-1)\geq N$, we can use 
        Hölder's inequality with exponents $N/(N-1)$ and $N$,  and Theorem \ref{Sobolev negative} to obtain
        \begin{align*}
			\dint W|u|^{q-1}|\nabla u|dx & \leq  \left(\dint W_{x_N}|u|^{(q-1)N/(N-1)}dx\right)^{(N-1)/N} \|u\|_{\mathcal{D}_W^-} \leq C_6 \|u\|_{\mathcal{D}_W^-}^q
		\end{align*}
        and the result follows from \eqref{JP-20}, as in the case $1<p<N$.
        \end{proof}

\section{Final comments}\label{Final comments} 

We devote this final section to presenting some examples of applications of our abstract setting as well as some possible directions of future research.

\subsection{Applications to nonlinear PDE's}
As it is well known, the framework of Sobolev spaces is a cornerstone of the modern functional analytic approach to Partial Differential Equations (PDEs).
To illustrate this, we consider a  zero-mass problem with Neumann boundary conditions of the form 
\begin{equation}\label{Problem-Neumann}
\left\{
\begin{array}{rcll}
-\operatorname{div}\!\big(\rho(x)\,|\nabla u|^{p-2}\nabla u\big)&=&\lambda a(x)|u|^{p-2}u+ f(x,u), & \text{in } \Omega, \medskip\\ 
\rho(x)\,|\nabla u|^{p-2}(\nabla u \cdot \nu) &=&\mu b(x')|u|^{p-2}u+ g(x',u), & \text{on } \partial\Omega.
\end{array}
\right.
\end{equation}
where $\nu=\nu(x')$ is the outward normal vector at $x'\in\partial\Omega$. The spatially varying coefficient $\rho: \Omega \to \mathbb{R}$, the non-negative functions $a,\,b$ and the continuous nonlinearities $f:\Omega \times \mathbb{R} \to \mathbb{R}$, $g:\partial\Omega \times \mathbb{R} \to \mathbb{R}$ verify some mild growth conditions which enable us to use critical point theory. 

More specifically, we would like to look for critical points of the energy functional associated with the equation. Formally, it has the form,
\[
I(u)
 := \frac{1}{p}\int_\Omega \rho(x)\,|\nabla u|^p\,dx - \frac{\lambda}{p} \int_\Omega a(x) |u|^p\,dx  - \frac{\mu}{p}\int_{\partial\Omega} b(x')\,|u|^p\,d\sigma
   - J(u)
\]
with
$$
J(u) := \int_\Omega F(x,u)\,dx
   + \int_{\partial\Omega} G(x',u)\,d\sigma
$$
where $F(x,t) := \int_0^t f(x,\tau)d\tau$,  $G(x',t):=\int_0^t g(x',\tau)d\tau$ are the primitives of $f(x,\cdot)$, $g(x',\cdot)$, respectively, and $d\sigma$ stands for the surface element on $\partial\Omega$.

The idea now is guarantee that all the above integrals are well defined and the functional $I$ has good linking properties.

\begin{example}\label{Example1} 
Suppose $W$ satisfies the assumptions of Theorem \ref{trace-positive} and  $\rho$ is such that
\[
\rho(x) \geq C_0\,\frac{[W(x)]^p}{[W_{x_N}(x)]^{p-1}}, 
\qquad \forall\, x \in \Omega.
\]
As a first approach, we assume the Sobolev case $1<p<N$ and that there exist $C_1>0$, $0 \leq s \leq p$ and $q \in [p,p_*]$ such that
$$
|f(x,t)| \leq C_1 [W_{x_N}(x)]^{s/p}|t|^{q(s)}, \qquad \forall \, (x,t) \in \Omega \times \mathbb{R},
$$
and
$$
|g(x',t)| \leq C_1 W(x',\psi(x'))|t|^q, \qquad \forall  \,(x',t) \in\partial\Omega \times \mathbb{R}.
$$
Under these condition, we may use standard calculations to show that $J \in C^1(\mathcal{D}_W^+,\mathbb{R})$.

Concerning the $p$-quadratic part of the functional $I$, we assume that
$$
0 \leq a(x) \leq a_0 W_{x_N}(x), \quad 0 \leq b(x')\leq b_0W(x'),\qquad \forall\, \,x\in \Omega,\, x'\in \partial\Omega.
$$
From our embedding results, we obtain $C_2>0$ such that
$$
\int_\Omega a(x) |u|^p dx \leq C_2\|u\|_{\mathcal{D}_W^+}^p, \quad \int_{\partial\Omega} b(x')|u|^p d\sigma \leq C_2\|u\|_{\mathcal{D}_W^+}^p,
$$
for any $u \in \mathcal{D}_W^+$.

We now define 
$$
\|u\|_E:= \left(  \int_\Omega \rho(x)\,|\nabla u|^p\,dx -  \mu_1 \int_\Omega a(x) |u|^p\,dx  - \mu_2 \int_{\partial\Omega} b(x')\,|u|^p\,d\sigma \right)^{1/p}
$$
and use all the above inequalities to write
$$
\|u\|_E^p \geq \Big( C_0 - (\lambda + \mu)C_2\Big) \|u\|_{\mathcal{D}_W^+}^p,\qquad \forall\, u \in \mathcal{D}_W^+.
$$
So, if the parameters $\lambda,\,\mu$ are small, our functional behaves like
$
I(u) \sim  \|u\|_E^p - J(u)$, for any $u \in E$, where $E$ is the Banach space $\mathcal{D}_W^+$ endowed with the (equivalent) norm $\|\cdot\|$.

We can handle the borderline case $p=N$ in the same way, just replacing the growth condition on $f$ and $g$ by
$$
|f(x,t)| \leq C_1 W_{x_N}(x)|t|^{q_1},\quad |g(x',t)|  \leq C_1W(x',\psi(x'))|t|^{q_2},
$$
for exponents $q_1,\,q_2 \in [N,+\infty)$.
\end{example}

\begin{example}    
Another application of our results is the study of  nonlinear elliptic problems with Robin boundary conditions of the form
\begin{equation} \label{mineiro1}
\left\{
\begin{array}{rcll}
-\operatorname{div}\!\big(\rho(x)\,|\nabla u|^{p-2}\nabla u\big) &=& f(x,u), & \text{in } \Omega, \medskip\\
\rho(x')\,|\nabla u|^{p-2}(\nabla u\!\cdot\!\nu) + b(x')|u|^{p-2}u &=& g(x',u), & \text{on } \partial\Omega.
\end{array}
\right.
\end{equation}
The main difference here is that the non-negative potential $b$ appears on the left-hand side of the boundary condition. 

In this case, considering the weighted Sobolev space $E$ defined by the completion of $C_0^\infty(\mathbb{R}^N)$ with repecto to the norm 
$$
\|u\|_E:=\left(\dint \frac{W^p}{[W_{x_N}]^{p-1}} |\nabla u|^pdx+\dint W_{x_N}|u|^pdx\right)^{1/p},
$$
we see that in the space $E$ the Friedrich norm  
$$
    \|u\|_F:=\left(\dint \frac{W^p}{[W_{x_N}]^{p-1}}|\nabla u|^pdx+\int_{\mathbb{R}^{N-1}} W(x',\psi(x'))|u(x',\psi(x')|^pdx'\right)^{1/p},
$$
is equivalent to $\|\cdot\|_E$. Thus, we can assume similar conditions as in the Example~\ref{Example1} in order to address results of existence and Liouville type results for this class of problems, see for instance \cite{zbMATH05077864}.
\end{example}

\begin{example}
Of course, we can also consider  the case where $W$ is decreasing. More precisely, by combining the Hardy-type inequality \eqref{Hardy-Geral-Grafico-II} with the Sobolev embedding established for the decreasing case, we can treat problem \eqref{mineiro1}
with the weighted function $\rho$ satisfying
	$$
	\rho(x)\geq  C_0\,\frac{[W(x_N)]^p}{[-W_{x_N}(x_N)]^{p-1}}, \qquad \forall\, x \in \Omega.
	$$
    We point out that in this case we are not able to consider the Neumann boundary condition, that is, $b \equiv 0$. Nevertheless, it is possible to take $\rho \equiv 1$ in this setting (see, for instance, Example~\ref{Example-Decreassing}).
\end{example}

\subsection{Further Research Directions}

\begin{itemize}
\item Is the constant $$c(\gamma,p)=\left(\frac{\gamma -p+1}{p}\right)^{p}$$ in the inequality \eqref{Hardy-singular} optimal and/or achieved? \\

\item As already mentioned, using inequality \eqref{Hardy-singular} together with a scaling argument, 
it is natural to expect that the following weighted Hardy--Sobolev--Maz'ya type inequality holds:
\begin{align*}
C_0
\left(
    \int_{\mathbb{R}^N_+} x_N^{\frac{N\gamma}{N-p}} |u|^{p^*} \, dx
\right)^{\!p/p^*}
 \leq
\int_{\mathbb{R}^N_+} x_N^{\gamma} |\nabla u|^p \, dx -
\frac{1}{(\gamma - p + 1)^{p}}
\int_{\mathbb{R}^N_+} \frac{|u|^p}{x_N^{\,p-\gamma}} \, dx,
\end{align*}
for all $u \in C_0^\infty(\mathbb{R}^N_+)$, $\gamma > p-1$ and some constant $C_0 = C_0(N,p,\gamma)$. Is this really true?\\

\item Is it possible to prove Theorem~\ref{Sobolev p=N} when the weight $W$ depends also of $x'$ with no additional condition, see Remark~\ref{chuva-sem-x'}?\\

\item The inequality \eqref{Hardy-Type II} in Theorem~\ref{Hardy-Type II} has  the form 
\begin{equation} \label{AB}
    \begin{aligned}
        \int_{\Omega} [-W_{x_N}]|u|^pdx & \leq   A \int_{\Omega} \frac{W^p}
         {\left[-W_{x_N}\right]^{p-1}}|\nabla u|^pdx \\ & +   B	\int_{\mathbb{R}^{N-1}} W(x',\psi(x'))|u(x',\psi(x'))|^pdx'
    \end{aligned}    
\end{equation}
 where $A$ and $B$ are positive constants.  As is well known, see for instance \cite{Druet-Hebey},  the first best constant associated to \eqref{AB} is defined by
  $$
  A_{\mathrm{opt}}:=\inf \Bigl\{A\in \mathbb{R}: \mbox{there exists }  B\in\mathbb{R}  \mbox{ such that } \eqref{AB}\mbox{ is valid} \Bigr\}
  $$
whose corresponding inequality is given by 
\begin{equation}\label{Aoptimal}
        \begin{aligned} \int_{\Omega} [-W_{x_N}]|u|^pdx & \leq A_{\mathrm{opt}}\int_{\Omega} \frac{W^p}
         {\left[-W_{x_N}\right]^{p-1}}|\nabla u|^pdx \\ & +  B	\int_{\mathbb{R}^{N-1}} 
         W(x',\psi(x'))|u(x',\psi(x'))|^pdx'.
         \end{aligned}
  \end{equation}
Now, it is natural to investigate the \emph{second-best constant} associated with \eqref{Aoptimal}, defined by 
\[
B_{\mathrm{opt}}
   := \inf\Bigl\{ B \in \mathbb{R} : \eqref{Aoptimal} \text{ holds with this value of } B \Bigr\}.
\]
Thus, one may investigate the precise value of optimal constants $A_{\mathrm{opt}}$ and $B_{\mathrm{opt}}$ in the above inequality.\\

\item If we assume that $W_1,W_2$ are positive weights, can we prove an inequality as below
$$
    \int_{\Omega} W_1 |u|^p + C_1\int_{\partial\Omega} W_2 |u|^p\leq C_2\int_{\Omega}  |\nabla u|^pdx, \quad \forall u\in C_0^\infty(\mathbb{R}^N),
$$
for some constants $C_1, C_2>0?$
Can we prove the same result above when the co-dimension is greater than one?\\

\item We observe that the monotonicity of \( W \) plays a decisive role in determining the position of the boundary term in our inequalities. A natural question arises: is this monotonicity assumption necessary in order to obtain results of this type?\\

\item In view of the Sobolev embedding in the borderline case $p = N$, a natural question is whether a Trudinger--Moser type inequality holds in the spaces $\mathcal{D}_W^{+}$ and $\mathcal{D}_W^{-}$.

\end{itemize}
\bigskip

\begin{flushleft}
	{\bf Funding:}  
	J. M. do \'O acknowledges partial support from CNPq through grants 312340/2021-4, 
    409764/2023-0, 443594/2023-6, CAPES MATH AMSUD grant 88887.878894/2023-00. 
    M. Furtado was partially supported by CNPq/Brazil and FAPDF/Brazil.
    E. Medeiros acknowledges partial support from CNPq through grant 310885/2023-0 
    and Para\'iba State Research Foundation (FAPESQ), grant no 3034/2021. 
    J. Ratzkin is partially 
    supported by the Deutsche Forschungsgemeinschaft (DFG) through grant \# 561401741.\\
	{\bf Ethical Approval:}  Not applicable.\\
	{\bf Competing interests:}  Not applicable. \\
	{\bf Authors' contributions:}    All authors contributed to the study conception 
    and design. All authors performed material preparation, data collection, and analysis. 
    The authors read and approved the final manuscript.\\
	{\bf Availability of data and material:}  Not applicable.\\
	{\bf Ethical Approval:}  All data generated or analyzed during this study are 
    included in this article.\\
	{\bf Consent to participate:}  All authors consent to participate in this work.\\
	{\bf Conflict of interest:} The authors declare that they have no conflict of interest. \\
	{\bf Consent for publication:}  All authors consent for publication. \\
\end{flushleft}

\bibliographystyle{abbrv}
\bibliography{bibliography}

\end{document}